\documentclass[a4paper]{amsart}
\usepackage[utf8]{inputenc}
\usepackage[T1]{fontenc}
\usepackage{lmodern}
\usepackage{amsmath, amsthm, amssymb, amsfonts}
\usepackage{amsxtra} %for sptilde
\usepackage[all]{xy}
\usepackage{nicefrac,mathtools}
\usepackage[shortlabels, inline]{enumitem}
\usepackage{microtype}
\usepackage{hyperref}
\usepackage{graphicx}
\usepackage[all]{xy}
\usepackage{xcolor}
\usepackage{tikz-cd}
\usepackage[lite]{amsrefs}
\usepackage{thmtools}

\numberwithin{equation}{section}
\theoremstyle{plain}
\newtheorem{theorem}[equation]{Theorem}

\newtheorem{lemma}[equation]{Lemma}

\newtheorem{proposition}[equation]{Proposition}
\newtheorem{corollary}[equation]{Corollary}

\theoremstyle{definition}
\newtheorem{definition}[equation]{Definition}

\theoremstyle{remark}
 \newtheorem{remark}[equation]{Remark}
 \theoremstyle{remark}
 \newtheorem{observation}[equation]{Observation}
\newtheorem{example}[equation]{Example}

\newcommand*{\Cont}{\mathrm C}%continuous functions
\newcommand*{\Contc}{\mathrm{C_c}}%continuous functions with compact
\newcommand{\Index}{\Lambda}
\newcommand*{\Cc}{\mathrm{C_c}}%continuous functions with compact support
\newcommand{\Contz}{\mathrm{C_0}}
\newcommand{\base}[1][G]{{#1}^{(0)}}
\newcommand{\supp}{{\textup{supp}}}

\newcommand*{\Hilm}[1][H]{\mathcal #1}% Hilbert module
\newcommand*{\Hils}[1][H]{\mathcal #1}% Hilbert module
\newcommand*{\Mult}{\mathcal M}%multiplier algebra
\newcommand{\Bound}{\mathbb B}
\newcommand*{\Cst}{\mathrm{C}^*}%C*-algebra
\newcommand{\A}{\mathcal{A}}
\newcommand{\B}{\mathcal{B}}
\newcommand*{\inpro}[2]{\langle#1, #2\rangle}% right inner products
\newcommand*{\binpro}[2]{\bigl\langle#1, #2\bigr\rangle}% right inner
\newcommand*{\Binpro}[2]{\Bigl\langle#1, #2\Bigr\rangle}% right inner
\newcommand{\norm}[1]{\lvert\!\lvert #1\rvert\!\rvert}
\newcommand{\C}{\mathbb{C}}
\newcommand{\R}{\mathbb{R}}
\newcommand{\N}{\mathbb{N}}
\newcommand{\iso}{\simeq}

\newcommand*{\Star}{$^*$}
\newcommand*{\nb}{\nobreakdash}
\newcommand*{\defeq}{\mathrel{\vcentcolon=}}
\newcommand{\inverse}{^{-1}}

\newcommand{\etale}{{\'e}tale}

\title[Fell bundle
\(\mathrm{C}^*\)-algebras]{\(\mathrm{C}^*\)-algebras of Fell bundles
  over étale groupoids}

\author{Rohit Dilip Holkar} \email{rohit.d.holkar@gmail.com}

\author{Md Amir Hossain}
\email{mdamirhossain18@gmail.com}

\address{Department
	of Mathematics, Indian Institute of Science Education and Research
	Bhopal, Bhopal Bypass Road, Bhauri, Bhopal 462 066, Madhya Pradesh,
	India.}

      \keywords{Representations; étale groupoids; Fell bundles; \(\Cst\)-algebras.}  
      \thanks{\emph{Subject class.} 47L55, 46L55,
        22A22.}

\begin{document}
      \begin{abstract}
        We describe a construction for the full
        \(\mathrm{C}^*\)-algebra of a possibly unsaturated Fell
        bundle over a possibly non-Hausdorff locally compact étale
        groupoid without appealing to Renault's disintegration
        theorem. This construction generalises the standard
        construction given by Muhly and Williams.
      \end{abstract}

      \maketitle

      \section*{Introduction}
      \label{sec:intro}

      The \(\Cst\)\nb-algebra of a locally compact groupoid was
      defined by Renault~\cite{Renault1980Gpd-Cst-Alg} as the
      completion of \(\Cc(G)\) with respect to the full norm. The full
      norm of \(f\in \Cc(G)\) is given by supremum of
      \(\norm{\pi(f)}\) over all \(I\)\nb-norm bounded representation
      \(\pi\) of \(\Cc(G)\). Exel
      in~\cite{Exel2008Inverse-semigroup-combinotorial-Cst-alg}
      introduced a new approach to defining the \(\Cst\)\nb-algebras
      of \(\etale\) groupoid and proved that the algebra he defined is
      same as that of Renault's. Exel worked with \emph{non-Hausdorff}
      \etale\ groupoids.

      In~\cite{Sims-Szabo-Williams2020Operator-alg-book}
      and~\cite{Sims2017Etale-gpd}, Sims prove that for an \etale\
      groupoid any representation of \(\Cc(G)\) is \(I\)\nb-norm
      bounded and continuous in inductive limit topology. Clark and
      Zimmerman~\cite{Clark-Zimmerman2022A-steinberg-Appro-to-etale-gpd-alg}
      proved similar results as Sims for second countable groupoids
      which can be, possibly, \emph{non-Hausdorff}.

      Fell bundles over locally compact groupoids seems to be first
      introduced by
      Yamagami~\cite{Yamagami1990On-primitive-ideal-spac-gpd} and
      Kumjian~\cite{Kumjian1998Fell-bundles-over-gpd}. The Fell
      bundles over groupoids is a natural generalisation of
      \(\Cst\)\nb-dynamical system. One may construct Fell bundles
      from group actions on
      \(\Cst\)\nb-algebras~\cite{Exel1997Twisted-Partial-Acts-Fell-Bundle}. Conversely,
      Exel~\cite{Exel1997Twisted-Partial-Acts-Fell-Bundle} shows that
      a certain class of Fell bundles essentially corresponds to
      twisted partial actions of locally compact groups on
      \(\Cst\)\nb-algebras.  As a consequence, the
      \(\Cst\)\nb-algebras of \emph{nice} Fell bundles constitute a
      large class of \(\Cst\)\nb-algebras associated with
      \(\Cst\)\nb-dynamical systems.

      Turning our attention to \(\Cst\)\nb-algebras of Fell bundles,
      the \(\Cst\)\nb-algebra of a Fell bundle is typically defined as
      the completion of \(\Cc(G;\A)\) with respect to the full norm;
      the existence of this full norm needs a justification. Muhly and
      Williams~\cite{Muhly-Williams2008FellBundle-ME} offer an
      exhaustive construction of the (full) \(\Cst\)\nb-algebra of a
      Fell bundle over a \emph{locally compact groupoid with a Haar
        system}. Kumjian~\cite{Kumjian1998Fell-bundles-over-gpd}
      describes the reduced \(\Cst\)\nb-algebra of a Fell bundle.

      We indent to encounter a technical issue in the construction of
      the full \(\Cst\)\nb-algebra of a Fell bundle. The proof of
      Renault's disintegration
      in~\cite{Muhly-Williams2008FellBundle-ME} uses the fact that the
      Fell bundle is \emph{saturated}; the Hausdorffness of the
      underlying groupoid is also used in the proof. On the other
      hand, the partial actions of groupoids provide a large class of
      \emph{unsaturated} Fell bundles~\cite{Exel2017PDS-Fell-bundles,
        Exel1997Twisted-Partial-Acts-Fell-Bundle,
        Anantharaman-D:2020Partial-Action-of-Gpd}. We propose a method
      to construct the \(\Cst\)\nb-algebras of these
      \emph{unsaturated} Fell bundles over \emph{\etale}
      groupoids. This construction basically uses the convolution
      \Star-algebra of the groupoid which can be defined without
      bothering about the saturation of the Fell bundle. Even the
      proofs involved are algebraic in nature. Therefore, the
      saturation of the Fell bundle does not come into the picture
      while construction the \(\Cst\)\nb-algebra of the Fell
      bundle. In fact, since the convolution algebra can defined for a
      nonHausdorff groupoid, the proposed construction works for
      unsaturated Fell bundles over non\nb-Hausdorff \etale{} groupoids.
     
      For a Fell bundle \(p\colon \A \to G\) over an \etale\ groupoid,
      one can consider four types of representations
      of~\(\Cont(G;\A)\) on a Hilbert space: \Star-representations (no
      continuity condition imposed), \Star-representations continuous
      in the inductive limit topology or the \(I\)\nb-norm, and the
      pre-representations; see page~\pageref{it:alg-rep} for
      details. The ideas for proving that the classes of first three
      representations are are same, are coming from the construction
      of the \(\Cst\)\nb-algebra of the full \(\Cst\)\nb-algebra of an
      \etale\ groupoid introduced by
      Exel\cite{Exel2008Inverse-semigroup-combinotorial-Cst-alg},
      Sims~\cite{Sims-Szabo-Williams2020Operator-alg-book} and
      Clark--Zimmerman~\cite{Clark-Zimmerman2022A-steinberg-Appro-to-etale-gpd-alg};
      the motivation of this article comes from these works . Showing
      that the fourth representations are equivalent to any of the
      earlier ones requires more work. We do not use the
      Disintegration theorem. We show that our construction of
      \(\Cst(G;\A)\) is equivalent to that of
      Muhly--Williams'~\cite{Muhly-Williams2008FellBundle-ME} when the
      Fell bundle is saturated and underlying locally compact groupoid
      is Hausdorff and second countable.

      \medskip

      \paragraph{\itshape Structure of the article:}

      Section~\ref{sec:prelim} establishes basic ideas and definitions
      about groupoids, bundle of \(\Cst\)\nb-algebras and Fell bundles
      over groupoids.

      Section~\ref{sec:rep-of-Cc-G-A}, discusses the four notions of
      representations of~\(\Cc(G;\A)\). The equivalence of the first
      three norms is discussed in Section~\ref{sec:pre-rep}.
      Section~\ref{sec:equiv-three-repr} contains the proof that
      pre-represetations are also equivalent to the other three;
      Theorem~\ref{thm:norm-equi} summarises the main results. Many
      remarks after this theorem are given which point out certain
      properties of~\(\Cst(G; \A)\). Examples are discussed at the end.

      \medskip

\paragraph{\itshape Acknowledgement:}
This work was supported by SERB's grant MTR/2020/000198 of the first
author and the CSIR grant~09/1020(0159)/2019-EMR-I of the second
one. We are thankful to the funding institutions.

\medskip

\section{Preliminaries}
\label{sec:prelim}

We assume that the reader is familiar with locally compact \etale\
groupoids and construction of their \(\Cst\)\nb-algebras.  Our
references for locally compact groupoids and their
\(\Cst\)\nb-algebras
are~\cite{Exel2008Inverse-semigroup-combinotorial-Cst-alg,
  Renault1980Gpd-Cst-Alg, Sims-Szabo-Williams2020Operator-alg-book}.

The \emph{spaces} considered in this article are locally compact and
Hausdorff, and groupoids are mostly locally compact (but not
necessarily Hausdorff). For the sake of clarification, we call a space
\emph{locally compact} if every point of the space has a Hausdorff
neighbourhood with the property that every open cover of the
neighbourhood admits a finite subcover.

For a groupoid~\(G\), we denote its space of units by \(G^{(0)}\); the
range and the source maps by~\(r\) and~\(s\), respectively; and the
set of composable pairs,
\(\{(\gamma,\eta)\colon G\times G : s(\gamma)=r(\eta)\}\), by
\(G^{(2)}\). For \(\gamma\in G\), \(\gamma\inverse\) denotes its
inverse. The groupoid~\(G\) is called topological if it has a topology
that makes the range, source, inversion and multiplication maps
continuous when \(\base\subseteq G\) and
\(G^{(2)}\subseteq G\times G\) are given subspace topologies.

\subsection{\'Etale\ groupoids}
\label{sec:etale-groupoids}

A topological groupoids~\(G\) is called \emph{{\etale}} if its space
of units~\(\base\) is locally compact Hausdorff and the range map
(equivalently, the source map) \(G\to \base\) is a local homeomorphism
in the subspace topology~\(\base \subseteq G\). Thus
\(\base\subseteq G\) is open with the subspace topology.

Assume that~\(G\) is an \etale\ groupoid. Then the space of units
\(\base \subseteq G\) is an open subspace~\cite[Proposition
3.2]{Exel2008Inverse-semigroup-combinotorial-Cst-alg}. The space of
units is closed in~\(G\) \emph{iff} \(G\) is Hausdorff~\cite[Lemma
8.3.2]{Sims-Szabo-Williams2020Operator-alg-book}.  Since the groupoid
is locally homeomorphic to its space of units, which is locally
compact Hausdorff, the groupoid itself is locally compact hence
locally Hausdorff. In fact, every open Hausdorff subspace of~\(G\) is
locally compact~\cite[Proposition
3.7]{Exel2008Inverse-semigroup-combinotorial-Cst-alg}.

Since~\(G\) is not Hausdorff in general, \(\Contc(G)\) is defined in
the sense of
Connes~\cite{Connes1982A-Survey-of-folliantion-and-OP-Alg}. Thus
\(\Contc(G)\) is span of those functions~\(f\) on~\(G\) which are
supported on a compact set in an open Hausdorff subset~\(V\) of~\(G\)
and~\(f|_V\) is continuous. Thus, the functions in~\(\Contc(G)\) need
not be continuous. Nonetheless, \(\Contc(G)\) is an involutive
algebra.

A particular type of basic open sets, called bisections, in an \etale\
groupoid prove very useful in computation. An open subset
\(U\subseteq G\) is called a \emph{bisection} if the restrictions of
the range and source maps to~\(U\) are homeomrphisms on open subsets
of~\(\base\).  Clearly, every bisection is locally compact Hausdorff
subspace of~\(G\). And the topology of \etale\ groupoid~\(G\) has a
basis consisting of bisections; as a consequence \(\Contc(G)\) can be
realised as a spanned of \(\Contc(U)\) for
bisections~\(U\)~\cite[Proposition
3.10]{Exel2008Inverse-semigroup-combinotorial-Cst-alg}. We denote the
basis consisting of bisections by~\(G^{\mathrm{op}}\).

Continuing the last discussion, since \(r\) is a local homeomorphism,
it follows that \(G^x\) is discrete for all \(x\in G^{(0)}\). Let
\(\lambda_x\) be a counting measure on \(G^x\). The family of measure
\(\{\lambda_x\}_{x\in G^{0}}\) constitutes a Haar system for the
{\'e}tale groupoid \(G\),
see~\cite[Defintion~I.2.2]{Renault1980Gpd-Cst-Alg}.

\subsection{$\Cst$-algebras and Hilbert modules}
\label{sec:cst-algebras-hilbert}

Our references for \(\Cst\)\nb-algebras
are~Davidson's~\cite{Davidson1996Cst-by-Examples} book and that of
Murhpy's~\cite{Murphy-Book}; for Hilbert modules we refer to Lance's
monograph~\cite{Lance1995Hilbert-modules}.

Let \(A\) be a \(\Cst\)\nb-algebra. By a (right) Hilbert module
over~\(A\) we mean a Banach space~\(\Hilm\) that is a right
\(A\)\nb-module equipped with an \(A\)\nb-valued \emph{inner product}
\(\inpro{}{}_A\). This inner product is an \(A\)\nb-valued
\emph{positive definite sesquilinear form} on~\(\Hilm\) that is linear
in the second variable and antilinear in the first one. The positive
definiteness has the standard \(\Cst\)\nb-algebraic meaning here,
namely, \(\inpro{a}{a}\in A\) is positive and equals~\(0\) \emph{iff}
\(a=0\). It is assumed that the Banach norm on~\(\Hilm\) is given
by~\(a\mapsto \norm{\inpro{a}{a}}^{1/2}\) where~\(a\in \Hilm\). A left
Hilbert \(A\)\nb-module is defined similarly. We reserve the name
\emph{Hilbert} module for a \emph{right} Hilbert module; a left one
shall be explicitly pointed out.

\begin{example}
  Let \(A\) be a \(\Cst\)\nb-algebra and \(I\subseteq A\) a closed
  ideal. Then \(\inpro{a}{b}_A \defeq a^*b\) and
  \(_A\inpro{a}{b} \defeq ab^*\), respectively, define left and right
  \(A\)\nb-valued inner products on~\(I\). In fact they make~\(I\)
  \(A\)\nb-module. With these inner products \(I\) also becomes
  Hilbert \(A\)-\(A\)\nb-module. This Hilbert module is not full
  unless \(I=A\).
\end{example}

\begin{lemma}[{\cite[Page 5]{Lance1995Hilbert-modules}}]
  \label{lem:appx-unit-for-full-Hilbert-mod}
  Assume that \(X\) is a Hilbert \(A\)\nb-module and \(I\subseteq A\)
  the closed ideal generated by \(\inpro{X}{X}\). Let \((u_i)_{i}\) be
  an approximate identity of~\(I\). Then \(\norm{ x u_i- x }\to 0\)
  for \(x\in X\).
\end{lemma}

\begin{lemma}
  \label{lem:positive-element}
  For an element~\(a\) of a unital \(\Cst\)\nb-algebra~\(A\),
  \(\norm{a}^2 - a^*a\) is a positive element.
\end{lemma}
\begin{proof}
  The \(\Cst\)\nb-subalgebra of~\(A\) generated by~1 and~\(a^*a\) is
  \(\Cont(\sigma(a^*a))\). As the spectral radius of~\(a^*a\)
  is~\(\norm{a}^2\), \(\norm{a}^2 - a^*a\) is a positive element in
  this \(\Cst\)\nb-subalgebra.
\end{proof}

\subsection{Bundles of \(\Cst\)-algebras}
\label{sec:upper-semic-bundl}

We refer the reader to Appendix~C
of~\cite{Williams2007Crossed-product-Cst-Alg} for the basics of the bundles of Banach spaces and
upper semicontinuous bundles of~\(\Cst\)\nb-algebras over locally compact Hausdorff spaces. However, we shall keep recalling some results whenever required. Following~\cite[Defintion
I.13.4]{Fell1988Representation-of-Star-Alg-Banach-bundles}, an
\emph{upper semicontinuous bundle of Banach spaces} is an open mapping of spaces \(p\colon \mathcal{B} \to X\) and following hold: each \emph{fibre}
\(\mathcal{B}_x \defeq p\inverse(x)\) for \(x\in X\) has a complex Banach
space structure and the \textrm{norm} from \(\mathcal{B}\to \R\) by \(b\mapsto \norm{b}\) is an upper semicontinuous function; the scalar multiplication\footnote{Fell writes
	an equivalent condition that \emph{looks} weaker,
	cf.~\cite[II.13.10]{Fell1988Representation-of-Star-Alg-Banach-bundles}}~\(\C
\times \mathcal{B} \to \mathcal{B}\) is also a continuous
function. Moreover, the addition is continuous from \(\mathcal{B}^{(2)} \to \mathcal{B}\) where
\(\mathcal{B}^{(2)}\) is the fibre product
\(\{(b,b')\in \mathcal{B}\times \mathcal{B} : p(b) = p(b')\}\). Finally, one more condition regarding the topology on~\(X\) and the norms on fibres is satisfied.
We call  \(\mathcal{B}\) is the
\emph{total space}, \(X\) is the \emph{base space} and~\(p\) is the
\emph{bundle map}. 

An upper
semicontinuous Banach bundle \(p\colon \A\to X \) is called an upper semicontinuous \(\Cst\)\nb-bundle if each of the 
fibre is a \(\Cst\)\nb-algebra, and the following conditions hold: the fibrewise multiplication is a continuous mapping
\(\A^{(2)} \to \A \) and the fibrewise involution, \(a \mapsto a^* \) from \(\A\to \A\) is continuous.
 The base spaces in all our bundles shall be assumed to be
locally compact and Hausdorff. It is customary to write either~\(p\)
or~\(\A\) instead of \(p\colon \A\to X\) when the base space and the
bundle map in a bundle are clear. An upper semicontinuous bundle of
\(\Cst\)\nb-algebras shall be simply called a bundle of
\(\Cst\)\nb-algebras or a \(\Cst\)\nb-bundle.

 Given a subspace \(Y\subseteq X\), \(p\colon \A|_Y \to Y\) denotes the
restricted bundle; note that the bundle map is still denoted by~\(p\).

Assume that a locally compact Hausdorff and second countable
space~\(X\) is given. A bundle of \(\Cst\)\nb-algebras
\(p\colon \A \to X\) is called \emph{separable} if each fibre in~\(p\)
is separable. In this case, the section algebra \(\Contz(X;\A)\) is
separable.

Given a bundle \(p\colon \A \to X\), a continuous
section has standard meaning and we shall simply call a
continuous section a \emph{section}. If \(K\subseteq X\) is compact, \(\Cont(K; \A|_K)\) is the
\(\Cst\)\nb-algebra of continuous sections of \(\A|_K \to K\).
A bundle \(p\colon \A \to X\) is said to have \emph{enough sections}
if for~\(x \in X\) and~\(a\in \A_x\), there is a section~\(f\)
with \(f(x)=a\). If the base space~\(X\) is locally compact and
Hausdorff, then the bundle \(p\colon \A\to X\) has enough
section~\cite{Hofmann1977Bundles-and-sheaves-equi-in-Cat-Ban}. In
general, Fell and Doran~\cite[Vol.\,1,
C.\,16]{Fell1988Representation-of-Star-Alg-Banach-bundles} prove the
existence of enough sections for an upper semicontinuous bundle of
Banach spaces over a paracompact base space.

For a bundle of \(\Cst\)\nb-algebras \(p\colon \A\to X\),
\(\Contz(X;\A)\) denotes the \(\Cst\)\nb-algebra of sections vanishing
at infinity.  \(\Contc(X; \A)\) denotes the dense two-sided ideal
of~\(\Contz(X;\A)\) consisting of continuous sections with compact
support.

\begin{lemma}[Proposition~1.7.1
  in~\cite{Dixmier1977Cst-Alg-Enlglish}] \label{lem:dixmier-approx-id}
  Let \(A\) be a \(\Cst\)\nb-algebra, and~\(m\) a two-sided ideal
  of~\(A\) which is dense in~\(A\). Then there is an increasing
  approximate identity of~\(A\) consisting of elements
  of~\(m\). If~\(A\) is separable, this approximate identity can be
  taken to be indexed by \(\{1,2, \dots\}\).
\end{lemma}

\begin{corollary}[Corollary of
  Lemma~\ref{lem:dixmier-approx-id}]\label{cor:cc-apprx-unit-general-case}
  Assume an upper semicontinuous \(\Cst\)\nb-bundle
  \(p\colon \A \to X\) over a locally compact Hausdorff space~\(X\) is
  given.  Then \(\Contz(X;\A)\) has a approximate identity consisting
  of elements in~\(\Contc(X;\A)\). Moreover, if the bundle is
  separable, then a countable approximate can be chosen.
\end{corollary}

\begin{lemma}
  \label{lem:local-cont-of-alg-rep}
  Let \(p\colon \A \to X\) be an upper semicontinuous bundle of
  \(\Cst\)\nb-algebras over a locally compact Hausdorff space. Assume
  that a nondegenerate \Star\nb-homomorphism
  \(L\colon \Contc(X;\A) \to \Bound(\Hils)\) is given. Then~\(L\) is
  continuous in the inductive limit topology on~\(\Contc(X;\A)\) and
  norm topology on \(\Bound(\Hils)\).
\end{lemma}

\begin{proof}
  Let \(\mathcal{C}\) denote the directed set consisting of compact
  subsets of~\(X\) as in
  Corollary~\ref{cor:cc-apprx-unit-general-case}. For
  \(K\in \mathcal{C}\), let~\(\Cont(K; \A|_K)\) be the Banach space
  consists of continuous section of~\(p\colon \A \to X\) vanishing
  outside~\(K\). The inductive limit topology on \(\Contc(X;\A)\) is
  induced by \(\{\Cont(K; \A|_K)\}_{K\in \mathcal{C}}\). Since
  each~\(\Cont(K;\A|_K)\) is in fact a \(\Cst\)\nb-algebra, the
  restriction of \(L|_{\Cont(K;\A|_K)}\) is a homomorphism of
  \(\Cst\)\nb-algebras, hence a contractive map.
  As~\(L|_{\Cont(K;\A|_K)}\) is continuous for every relatively
  compact open set~\(K \subseteq X\), so is~\(L\).
\end{proof}

Notice the proof of last lemma implies that for
\(f\in \Cont(K;\A|_K)\), \(\norm{L(f)}\leq \norm{f}_\infty\) for
\(L|_{\Cont(K;\A|_K)}\) is a contraction.

\subsection{Fell bundles}

By now, Fell bundles are established objects in Operator Algebras
which are being widely studied. Therefore, we shall \emph{briefly}
describe them. Our main references for Fell bundles
are~\cite{Muhly-Williams2008FellBundle-ME,
  Kumjian1998Fell-bundles-over-gpd}. Many of our notation and
elementary discussion of Fell bundles is adopted
from~\cite{Muhly-Williams2008FellBundle-ME}. We consider the
definition of Fell bundles from~\cite{Muhly-Williams2008FellBundle-ME}
for \etale\ groupoids as we defined; thus our groupoids are locally
compact and not necessarily Hausdorff. Our use of this definition is
justified if one sees the standard construction of Fell bundles from
partial dynamical systems. Moreover, the Fell bundles we consider are
not necessarily \emph{saturated}.

\begin{definition}[Fell bundle~\cite{Muhly-Williams2008FellBundle-ME}]
  \label{def:fell-bundle}
  A Fell bundle over an \etale\ groupoid~\(G\) is an upper
  semicontinuous bundle of Banach spaces \(p\colon \A\to G\) equipped
  with a continuous `multiplication' map \(m\colon \A^{(2)}\to \A\)
  and an involution map \(\A\to \A, a\mapsto a^*\) for \(a\in \A\)
  which satisfy the following axioms:
  \begin{enumerate}
  \item\label{item:Fell-1} \(p(m(a,b)) =p(a)p(b)\) for all
    \((a,b) \in \A^{(2)}\);
  \item\label{item:Fell-2} the induced map from
    \(\A_{\gamma_1}\times \A_{\gamma_2} \to \A_{\gamma_1 \gamma_2}\),
    \((a,b) \mapsto m(a,b)\) is bilinear for all
    \((\gamma_1, \gamma_2)\in G^{(2)}, a\in \A_{\gamma_1}\) and
    \(b\in \A_{\gamma_2}\);
  \item\label{item:Fell-3} \(m(m(a,b),c) =m(a, m(b,c))\) for
    \((a,b),(b,c)\in \A^{(2)}\);
  \item\label{item:Fell-4} \(\norm{m(a,b)} \leq \norm{a}\norm{b}\) for
    all \((a,b) \in \A^{(2)}\);
  \item\label{item:Fell-5} \(p(a^*)=p(a)^{-1}\) for all \(a\in \A\);
  \item\label{item:Fell-6} the induced map
    \(\A_\gamma \to \A_{\gamma^{-1}}\), \(a\mapsto a^*\) is conjugate
    linear for all \(\gamma \in G\);
  \item\label{item:Fell-7} \((a^*)^*=a\) for all \(a\in \A\);
  \item\label{item:Fell-8} \(m(a,b)^*=m(b^*,a^*)\) for all
    \((a,b)\in \A^{(2)}\);
  \item\label{item:Fell-9}
    \(\norm{m(a^*,a)}=\norm{m(a,a^*)}=\norm{a}^2\) for all
    \(a\in \A\);
  \item\label{item:Fell-10} \(m(a^*,a) \geq 0\) for all \(a\in \A\).
  \end{enumerate}
\end{definition}

\noindent Note that the last definition makes sense for non-\etale\
groupoids as well. In the above definiton \(\A^{2}\) is given by the set \(\{(a,b) : (s(p(a)), r(p(b))) \in G^{(2)}\}\). Given \((a,b) \in \A^{(2)}\), we use the
shorthand \(ab\) instead of \(m(a, b)\).  For two composition arrows
\(\gamma_1,\gamma_2\in G\), define the subset
\[
  \A_{\gamma_1} \cdot \A_{\gamma_2} = \textup{span} \{ab: a\in
  \A_{\gamma_1}, b\in \A_{\gamma_2}\} \subseteq \A_{\gamma_1\gamma_2}.
\]

\noindent The Fell bundle is called \emph{saturated} if
\(\A_{\gamma_1} \cdot \A_{\gamma_2}\) is dense in
\(\A_{\gamma_1 \gamma_2}\) for all
\((\gamma_1,\gamma_2) \in G^{(2)}\).

For \(x\in G^{(0)}\),
Conditions~\eqref{item:Fell-2}--\eqref{item:Fell-9} in
Definition~\ref{def:fell-bundle} make the fibre \(\A_x\) a
\(\Cst\)\nb-algebra with the obvious norm, multiplication and
involution. Furthermore, if \(\gamma \in G\) and \(a,b\in \A_\gamma\),
then \(\A_\gamma\) is a right Hilbert \(\A_{s(\gamma)}\)\nb-module
when equipped with the inner product
\(\inpro{a}{b}_{\A_{s(\gamma)}}\defeq a^*b\); and~\(\A_\gamma\) is a
left Hilbert \(\A_{r(\gamma)}\)\nb-module if the inner product is
defined by \({}_{\A_{r(\gamma)}}\inpro{a}{b} \defeq ab^*\). Thus, the Fell bundle
\(p\colon \A \to G\) is saturated if and only if~\(\A_\gamma\) is an
\(\A_{r(\gamma)}\)-\(\A_{s(\gamma)}\)\nb-imprimitivity bimodule for
each~\(\gamma \in G\).

For a subset ~\(X\subseteq G\), by \(\A|_X\) we denote the restriction
of~\(\A\) to~\(X\); but we still denote the bundle map~\(\A|_X\to X\)
by~\(p\) rather than~\(p|_{\A|_X}\).

We define \(\Contc(G; \A)\) analogous to \(\Contc(G)\) as
in~\cite{Connes1982A-Survey-of-folliantion-and-OP-Alg}. Thus, the
elements of \(\Contc(G;\A)\) are not necessarily continuous. With this
definition of \(\Contc(G;\A)\), the bundle~\(p\) has enough sections:
for this let \(\gamma\in G\) and \(\xi\in \A_\gamma\). For a locally
compact Hausdorff neighbourhood~\(X\) of~\(\gamma\) consider the
restricted bundle~\(\A|_X \to X\). Then this bundle has a
section~\(f\) which is continuous on~\(X\) has compact support
in~\(X\) and \(f(\gamma) = \xi\). Then \(f\in \Contc(G; \A)\) (in
fact, \(f\) belongs to the set which is typically denoted by
\(\Cont^0_{\mathrm{c}}(G;\A)\), see~\cite[Definition
3.9]{Exel2008Inverse-semigroup-combinotorial-Cst-alg} whose span
is~\(\Contc(G;\A)\)). In any case, we assume that our bundle has
enough sections.

\begin{lemma}
  Let \(G\) be a locally compact Hausdorff groupoid. Let
  \(p\colon \A \to G\) be a \emph{saturated} Fell bundle. Then each
  fibre of~\(\A\) is separable \emph{iff} for each \(x\in \base\) the
  fibre \(\A_x\) is a separable \(\Cst\)\nb-algebra.
\end{lemma}
\begin{proof}
  One implication is clear. For the converse, let \(\gamma\in
  G\). Then \(A_\gamma\) is an
  \(A_{r(\gamma)}\)-\(A_{s(\gamma)}\)-equivalence. Recall
  from~\cite[Corollary~1.1.25]{Jensen-Thomsen1991KK-theory-book} that
  a Hilbert module is countably generated \emph{iff} the
  \(\Cst\)\nb-algebra of the compact operators on it is
  \(\sigma\)\nb-unital. Since both~\(A_{r(\gamma)}\)
  and~\(A_{s(\gamma)}\) are separable, \(A_\gamma\)~is countably
  generated, hence separable.
\end{proof}

\noindent Last lemma says that, a \emph{saturated} Fell bundle over a
locally compact second countable Hausdorff groupoid is separable
\emph{iff} the \(\Cst\)\nb-bundle over the unit space is separable.

\section{Representations of \(\Contc(G;\A)\)}
\label{sec:rep-of-Cc-G-A}

Consider a Fell bundle \(p\colon \A \to G\) over an \etale\
groupoid~\(G\). As it is done for groups or groupoids, one can
construst the \emph{full} \(\Cst\)\nb-algebra \(\Cst(G;\A)\) of this
Fell bundle. One way to define this \(\Cst\)\nb-algebras is to
announce it the enveloping \(\Cst\)\nb-algebra of the
\(^*\)\nb-algebra \(\Contc(G;\A)\). The existence of the enveloping
algebra needs to be justified. Another way, similar to locally compact
groups, is to define it to be the universal \(\Cst\)\nb-algebra using
the unitary representation \emph{of the Fell
  bundle}. In~\cite{Muhly-Williams2008FellBundle-ME}, Muhly and
Williams define~\(\Cst(G;\A)\) in both the ways for a groupoid with a
\emph{Haar system}; they assume that the groupoid is locally compact
second countable and Hausdorff and the Fell bundle is saturated. In
the same article, they also prove Renault's integration and
disintegration theorems which establishes the equivalence of the two
methods of defining~\(\Cst(G;\A)\).

Third interesting, and practically useful way to construct
representations of Fell bundles is to use their
\emph{pre-representations}~\cite[Definitino
4.1]{Muhly-Williams2008FellBundle-ME}. All in all, in case of
\emph{\etale} groupoids, we may consider following four types of
\emph{representations} of the \Star\nb-algebra \(\Contc(G;\A)\) on a
Hilbert space~\(\Hils\):

\begin{enumerate}
\item\label{it:alg-rep} a nondegenerate \Star\nb-homomorphism
  \( \pi \colon \Cc(G;\A) \to \Bound(\Hils)\) (without any continuity
  condition);
\item\label{it:i-norm-rep} an \(I\)\nb-norm bounded nondegenerate
  \Star\nb-homomorphism \( \pi \colon \Cc(G;\A) \to \Bound(\Hils)\);
\item\label{it:ilt-rep} a nondegenerate \Star\nb-homomorphism
  \( \pi \colon \Cc(G;\A) \to \Bound(\Hils)\) continuous in the
  inductive limit topology;
\item\label{it:pre-rep} a \emph{pre-representation}
  \( \pi \colon \Cc(G;\A) \to \textup{Lin}(\Hils_0)\) on a dense
  subspace~\(\Hils_0\subseteq \Hils\).
\end{enumerate}

\noindent Our goal is proving Theorem~\ref{thm:norm-equi} which shows
that all four notions above are equivalent for Fell bundles over
\emph{\etale} groupoids \emph{without} appealing to the technical
machinery of the disintegration theorem
in~\cite{Muhly-Williams2008FellBundle-ME}. Nonetheless, extending a
pre-representation~\ref{it:pre-rep} to a one that is continuous in the
inductive limit topology~\ref{it:ilt-rep} involves technicalities.

The equivalence of all four representations above allows us to use any
one of these representations can be used to define the full
\(\Cst\)\nb-algebra \(\Cst(G;\A)\) for a Fell bundle \(\A\) over an
\etale\ groupoid~\(G\).

The upcoming content is planned as follows: we start the discussion by
defining some standard ideas which we have already mentioned
above. Then we quickly recall the \Star\nb-algebra \(\Contc(G;\A)\)
from~\cite{Muhly-Williams2008FellBundle-ME}. Then we recall the
notions of representations described above. In
Subsection~\ref{sec:equiv-three-repr}, we first prove the equivalence
of the three representations \ref{it:alg-rep}--\ref{it:ilt-rep} on
page~\pageref{it:alg-rep}. In in Section~\ref{sec:pre-rep}, we prove
equivalence of pre-representations with the rest ones.  \medskip

Now on, we fix~\(G\) to be a locally compact \etale\ groupoid (as in
\S~\ref{sec:etale-groupoids}) and a Fell bundle \(p\colon \A\to G\)
(the Fell bundle need not be \emph{saturated or separable}).  To start
with, make~\(\Contc(G;\A)\) into a \Star\nb-algebra using the
following convolution and involution:
\[
  f*g(\gamma)= \sum_{\alpha\beta=\gamma}f(\alpha)g(\beta) \quad
  \textup{and} \quad f^*(\gamma)=f(\gamma^{-1})^*
\]
where \(f,g\in \Contc(G;\A)\). It is a standard fact that last
operations make \(\Contc(G;\A)\) a \Star-algebra. The convolution is
continuous in the inductive limit topology
on~\(\Contc(G;\A)\). Since~\(G\) is {\etale}, \(\base\) is an open Hausdorff subset of \(G\). Then by definition of \(\Cc(G;\A)\), we can say that
\(\Contc(\base[G],\A|_{\base[G]})\) is a \Star\nb-subalgebra
of \(\Contc(G;\A)\).

Note that if \(f\) is a continuous section vanishing outside~\(\base\)
and \(g\in \Contc(G;\A)\), then
\begin{equation}
  \label{eq:conv}
  f*g(\gamma) = f(r(\gamma)) g(\gamma), \quad g*f(\gamma) = g(\gamma)
  f(s(\gamma))
\end{equation}
for \(\gamma\in G\). As \(f\circ r\) is continuous section on~\(G\),
\(f\circ r\cdot g\in \Contc(G; \A)\). Similarly,
\(g\cdot f\circ s\in \Contc(G; \A)\).

Define the \emph{\(I\)-norm} on \(\Contc(G;\A)\) by
\[
  \norm{f}_I= \max \biggl\{\sup_{x\in G^{(0)}}\sum_{\gamma\in
    r\inverse(x)} \norm{f(\gamma)},\sup_{x\in G^{(0)}}\sum_{\gamma\in
    s\inverse (x)} \norm{f(\gamma)}\biggr\}.
\]

\begin{example}
  \label{exa:I-norm-on-bisection}
  Let \(f\in \Contc(G;\A)\) be supported on a bisection~\(U\). Then
  \(\norm{f}_I = \norm{f}_\infty \). Because
  \[
    \sup_{x\in G^{(0)}}\sum_{\gamma\in r\inverse(x)} \norm{f(\gamma)}
    = \norm{f\circ s|_U\inverse}_{\infty} = \norm{f}_\infty =
    \norm{f\circ r|_U\inverse}_{\infty}= \sup_{x\in
      G^{(0)}}\sum_{\gamma\in s\inverse (x)} \norm{f(\gamma)}.
  \]
\end{example}

A \emph{representation} of the \Star\nb-algebra \(\Contc(G; \A)\) on a
Hilbert space \(\Hils\) is a non-degenerate \(^*\)-homomorphism
\(L\colon \Contc(G;\A)\to \Bound(\Hils)\) (without any continuity
condition on \(L\)). The representation~\(L\) is called
\emph{\(I\)-norm bounded} if \(\norm{L(f)}\leq \norm{f}_I\) for all
\(f\in \Contc(G;\A)\).

We shall also discuss representations
\(\Contc(G;\A) \to \Bound(\Hils)\) which are \emph{continuous in the
  inductive limit topology} on~\(\Contc(G; \A)\). Next is the
definition of a \emph{pre-representation}.

Let \(\Hils_0\) be a dense subspace of a Hilbert space \(\Hils\), and
denote by \(\textup{Lin}(\Hils_0)\) the collection of all linear
mappings of the vector space \(\Hils_0\) to itself.

\begin{definition}[{\cite[Definition~4.1]{Muhly-Williams2008FellBundle-ME}}]
  \label{def:pre-rep}
  A pre\nb-representation of \(\Cc(G;\A)\) on a Hilbert
  space~\(\Hils\) is a pair \((\Hils_0, L)\) consisting of a linear
  subspace \(\Hils_0 \subseteq \Hils\) and a homomorphism of complex
  algebras \(L\colon \Contc(G;\A) \to \textup{Lin}(\Hils_0)\) such
  that the following conditions holds for given
  \(\xi, \eta \in \Hils_0\):
  \begin{enumerate}
  \item (Continuity) \(f\mapsto \inpro{\eta}{L(f)\xi} \) is continuous
    in the inductive limit topology on \(\Contc(G;\A)\);
  \item (Adjoinatability)
    \( \inpro{\eta}{L(f)\xi} = \inpro{L(f^*)\eta}{\xi}\);
  \item \label{it:H00} (Nondegeneracy)
    \(\Hils_{00}\defeq \textup{span}\{L(f)\xi : f\in \Contc(G,\A), \xi
    \in \Hils_0 \}\) is dense in \(\Hils\).
  \end{enumerate}
\end{definition}
The nomenclature of properties above is not standard; we coin this
nomenclature as it facilitate latter writing.  We frequently
refer~\(L\) as the pre-representation on~\(\Hils\) assuming that the
representation subspace is~\(\Hils_0\). Moreover,~\(\Hils_{00}\) shall
have the same meaning as above.  Note that
\(\Hils_{00}\subseteq \Hils_0\), and the nondegeneracy condition
implies that~\(\Hils_{0}\) is dense in~\(\Hils\).

\subsection{Equivalence of the first three representations}
\label{sec:equiv-three-repr}

In current section, we prove equivalence of the representations
\ref{it:alg-rep}--\ref{it:ilt-rep} on page~\pageref{it:alg-rep}. Since
the inductive limit topology is weaker than the \(I\)\nb-norm
topology, a representation that is continuous in the inductive limit
topology is also continuous in the \(I\)\nb-norm topology. And a
representation in \(I\)\nb-topology is obviously a nondegenerate
\Star\nb-homomorphism. Therefore, among the definitions on
page~\pageref{it:alg-rep},
\ref{it:ilt-rep}\(\implies\)\ref{it:i-norm-rep}\(\implies\)
\ref{it:alg-rep}. The upcoming
Proposition~\ref{prop:three-reps-are-equi} proves that a nondegenerate
\Star\nb-homomorphism is, in fact, continuous in the inductive limit
topology; this justifies the equivalence of the first three notions of
representations.

The idea of this proof is coming from Section~3 in Exel's
article~\cite{Exel2008Inverse-semigroup-combinotorial-Cst-alg}; it can
also be found in Sims' book~\cite[\S
9.2]{Sims-Szabo-Williams2020Operator-alg-book}.

Consider the Fell bundle \(p\colon \A \to G\) over an \etale\
groupoid~\(G\) as fixed earlier. Recall from preliminaries
that~\(G^{\mathrm{op}}\) denotes the set of bisections in~\(G\). For
two bisections~\(U\) and~\(V\) in~\(G^{\mathrm{op}}\), define the sets
\[
  U^{-1} = \{u^{-1} : u\in U\} \quad \text{ and } \quad UV = \{ uv :
  u\in U, v\in V , (u,v) \in G^{(2)} \}.
\]
Then \(U\inverse, UV\in G^{\mathrm{op}}\); here \(VU\) could be
empty. In particular, it can be readily seens that the bisections
\(UU^{-1} = r(U)\) and \(U\inverse U = s(U)\) are subsets of
\(\base\). In fact,~\(G^{\mathrm{op}}\) is an inverse semigroup with
the inverses and products of subsets, see Proposition~2.2.4
in~\cite{PatersonA1999Gpd-InverseSemigps-Operator-Alg}.

For a bisection \(U\subseteq G\), we consider \(\Contc(U;\A|_U)\) as a
subset of \(\Contc(G;\A)\) in the obvious way. Following lemma is an
analogue of Proposition~3.11
in~\cite{Exel2008Inverse-semigroup-combinotorial-Cst-alg} for Fell
bundles.
  
\begin{lemma}\label{lem:convo-func:supp:bisec}
  Let \(U\) and \(V\) be bisections of \(G\).
  \begin{enumerate}
  \item
    \label{lem:convulu-supp} If \(f \in \Cc(U;\A|_{U})\) and
    \(g\in \Cc(V;\A|_{V})\), then \(f*g \in \Cc(UV;\A|_{UV})\).
  \item
    \label{lem:involu-supp} If \(f \in \Cc(U;\A|_{U})\), then
    \(f^* \in \Cc(U^{-1}; \A|_{U^{-1}})\).
  \item \label{lem:inv-with-adj} If \(f\in \Cc(U;\A|_{U})\), then
    \(f*f^*\) and \(f^**f \in \Cc(\base; \A|_{\base})\). In fact,
    \(f*f^* \in \Cc(r(U); \A|_{r(U)})\) and
    \(f^**f \in \Contc(s(U) ; \A|_{s(U)})\).
  \item \label{lem:uni-norm-bisec} If \(f\in \Cc(U;\A|_{U})\), then
    \(\norm{f*f^*}_\infty = \norm{f^**f}_\infty = \norm{f}_\infty^2\).
  \end{enumerate}
\end{lemma}

\noindent The proof of
(\ref{lem:convulu-supp})--(\ref{lem:inv-with-adj}) above lemma closely
follow that
of~\cite[Proposition~3.11]{Exel2008Inverse-semigroup-combinotorial-Cst-alg},
hence we skip them. The last claim follows from the following
observation: let \(U\) be a bisection and \(f\in \Contc(U;\A|_{U})\). If \(x\in r(U)\), then
\[
  \norm{f*f^*(x)} = \norm{f*f^*(\gamma \gamma\inverse)} =
  \norm{f(\gamma)}^2
\]
where \(\gamma\in r\inverse(x)\). And \(f*f^*(x) = 0\) if
\(x\notin r(U)\). The claim follows by taking supremum in the last
equation.

\begin{proposition}
  \label{prop:three-reps-are-equi}
  Let \(p\colon \A\to G\) be a Fell bundle over an \etale\
  groupoid~\(G\). Let \(L\colon \Contc(G;\A) \to \Bound(\Hils)\) be a
  nondegenerate \Star\nb-homomorphism. Then~\(L\) is continuous in the
  inductive limit topology on~\(\Contc(G;\A)\), and hence for the
  \(I\)\nb-norm on~\(\Contc(G;\A)\).
\end{proposition}

\begin{proof}
  First of all, note that since
  \(\Contc(\base[G]; \A|_{\base[A]})\subseteq \Contc(G;\A)\) is a
  \Star\nb-subalgebra, the
  restriction~\(L|_{\Contc(\base[G]; \A|_{\base[A]})}\) is continuous
  in the inductive limit topology
  on~\(\Contc(\base[G]; \A|_{\base[A]})\) due
  to~Lemma~\ref{lem:local-cont-of-alg-rep}.

  Now, let \(U\subseteq G\) be an open bisection. Let
  \(f\in \Contc(U;\A|_{U})\). So
  that~\(f^**f\in \Contc(\base[G]; \A|_{\base[G]})\) due to
  Lemma~\ref{lem:convo-func:supp:bisec}(3). With this in mind, we
  compute

\begin{equation}\label{equ:rep:contractive:bise}
  \norm{L(f)}^2 = \norm{L(f)^* L(f)} = \norm{L(f^**f)}\leq \norm{f^**f}_{\infty} = \norm{f}^2_{\infty}
\end{equation}

\noindent where the inequality follows from the remark after the proof
of Lemma~\ref{lem:local-cont-of-alg-rep}, and the last equality is due
to Lemma~\ref{lem:convo-func:supp:bisec}(\ref{lem:uni-norm-bisec}).

In general, let \(f\in \Cc(G;\A)\) with \(f = \sum_{i=1}^{n}f_i\)
where each \({U_i}\) is a bisection and \(f_i\in
\Cc(U_i;\A|_{U_i})\). Using Equation~\eqref{equ:rep:contractive:bise},
we see that
\[
  \norm{L(f)} \leq \sum_{i=1}^n\norm{L(f_i)} \leq \sum_{i=1}^n
  \norm{f_i}_\infty \leq n \norm{f}_\infty.
\]
\end{proof}

Next lemma, meant for a later use, concludes this section.

\begin{lemma}
  \label{lem:mult-by-bdd-sections}
  Let \((u_{\lambda})_{\lambda \in \Index}\) be an approximate unit
  of~\(\Contz(\base; \A|_{\base})\) where
  \(u_{\lambda}\in \Contc(\base; \A|_{\base})\), see
  Corollary~\ref{cor:cc-apprx-unit-general-case}.
  \begin{enumerate}
  \item For \(f\in \Contc(G;\A)\), \(u_{\lambda}*f\to f\) in the
    inductive limit topology.
  \item Let \(P\) be the vector subspace of~\(\Contc(G;\A)\) spanned
    by the set
    \(\Contc(\base;\A|_{\base})*\Contc(G;\A) \defeq \{\phi* f : \phi
    \in \Contc(\base;\A|_{\base}) \text{ and } f\in
    \Contc(G;\A)\}\). Then \(P\) is uniformly dense in
    \(\Contc(G;\A)\).
  \end{enumerate}
\end{lemma}
\begin{proof}
  (1): As \(\Contc(G;\A)\) is an inductive limit
  of~\(\Contc(U;\A|_U)\) for bisections~\(U\), proving the claim
  for~\(\Contc(U; \A|_U)\) is sufficient where~\(U\) is a
  bisection. Fix such~\(U\) and let \(f\in \Contc(U; \A|_{U})\).  Note
  that
  \(\supp(u_{\lambda}*f)\subseteq \supp(u_{\lambda})\supp(f) \subseteq
  \supp(f)\). We apply Lemma~\ref{lem:convo-func:supp:bisec}(4) for
  the bisection supported section \(u_{\lambda}*f -f\) and get
  \begin{multline*}
    \norm{u_{\lambda}*f-f}^2_\infty =
    \norm{(u_{\lambda}*f-f)*(u_{\lambda}*f-f)^*}_\infty
    = \norm{(u_{\lambda}*f-f)*(f^**u_{\lambda}-f^*)}_\infty \\
    = \norm{ u_{\lambda}*(f*f^*)*u_{\lambda} - u_{\lambda}*(f*f^*) -
      (f*f^*)*u_{\lambda} - f* f^*}_\infty.
  \end{multline*}
  As \(f*f^*\in \Contc(\base;\A|_{\base})\) and
  \((u_{\lambda})_{\lambda \in \Index}\) is an approximate unit
  of~\(\Contz(\base;\A|_{\base})\), the limit of last term is~\(0\). Therefore, \(u_{\lambda}*f\to f\) in the
  inductive limit topology.
  \smallskip

  \noindent(2) Follows immediately from first claim of current lemma
  as \(u_{\lambda}*f \in \Contc(\base;\A|_{\base})*\Contc(G;\A)\) for
  all \(f\in \Contc(G;\A)\) and \(\lambda \in \Index\).
\end{proof}

\noindent The second claim of last lemma can also be proved
using~\cite[Proposition~C24]{Williams2007Crossed-product-Cst-Alg}
without appealing to the approximate identity argument.

\subsection{Pre-representations of \(\Contc(G;\A)\) }
\label{sec:pre-rep}

A \Star\nb-representation of \(\Contc(G;\A)\) that is continuous in
the inductive limit topology is a obviously a
pre-representation. Conversely, assume that a pre-representation
\(L\colon \Contc(G;\A) \to \mathrm{Lin}(\Hils_{0})\) of
\(\Contc(G;\A)\) on a Hilbert space~\(\Hils\) is given. We shall
extend~\(L\) to a \Star\nb-representation of \(\Contc(G;\A)\)
on~\(\Hils\) that is continuous in the inductive limit topology in
Proposition~\ref{prop:pre-rep-to-rep-of-C0-X} of current section. This
shall show the all four notions of representation of \(\Contc(G;\A)\),
described in page~\pageref{it:alg-rep}, are equivalent.

Here is a notation: let~\(1_{\Mult(\Contz(\base; \A|_{\base}))}\)
denote the identity in the multiplier algebra
of~\(\Contz(\base;\A|_{\base})\). We fix an approximate
unit~\((u_{\lambda})_{\lambda \in \Index}\)
of~\(\Contz(\base;\A|_{\base})\) consisting of compactly supported
sections.

\begin{observation}\label{obs:pre-rep-appr-unit}
  Let \(\xi,\eta\in \Hils_{0}\), and \(f,g,k, h\in \Contc(G;\A)\),
  then
  \begin{align}
    \lim_{\lambda} \binpro{L(u_{\lambda}*f)\xi}{L(u_{\lambda}*k)\eta} &=
                                                                        \binpro{L(f)\xi}{L(k)\eta} \label{obs:pre-rep-computation-obs-1}\\
    \lim_{\lambda} \binpro{L(f*u_{\lambda})\xi}{L(k*u_{\lambda})\eta} &=
                                                                        \binpro{L(f)\xi}{L(k)\eta} \label{obs:pre-rep-computation-obs-2}\\
    \lim_{\lambda} \binpro{L(f*u_{\lambda}*g)\xi}{L(k*u_{\lambda}*h)\eta} &=
                                                                            \binpro{L(f*g)\xi}{L(k*h)\eta}.\label{obs:pre-rep-computation-obs-3}
  \end{align}
  For the first equation, we note that due to the adjointability of a
  pre-representation, the left hand side equals
  \(\lim_{\lambda}
  \inpro{L((u_{\lambda}*g)^**(u_{\lambda}*f))\xi}{\eta}\). Lemma~\ref{lem:mult-by-bdd-sections}(1)
  implies that \(u_{\lambda}*f\to f\) and \(g^**u_{\lambda}\to g\) in
  the inductive limit topology. Plus, the convolution is continuous in
  the inductive limit topology. Therefore, the continuity of a
  pre-representation implies that
  \(\lim_{\lambda}
  \binpro{L((u_{\lambda}*g)^**(u_{\lambda}*f))\xi}{\eta} = \binpro{L(
    g^* *f)\xi}{\eta} = \binpro{L(f)\xi}{L(g)\eta}\).
  Equation~\eqref{obs:pre-rep-computation-obs-2} is justified
  similarly.  The last one,
  Equation~\eqref{obs:pre-rep-computation-obs-3}, follows from
  Equation~\eqref{obs:pre-rep-computation-obs-2} by noticing that
  \begin{equation*}
    \binpro{L(f*u_{\lambda}*g)\xi}{L(k*u_{\lambda}*h)\eta} = \binpro{L(f*u_{\lambda})
      \bigl(L(g)\xi\bigr)}{L(k*u_{\lambda}) \bigl(L(h)\eta\bigr)}
  \end{equation*}
  where \(L(g)\xi, L(h)\eta\in \Hils_{00}\subseteq \Hils_0\).
\end{observation}

\begin{lemma}
  \label{lem:sqrt-of-q}
  For a section \(h\in \Contc(U;\A|_{U})\) where \(U\) is a bisection, define
  \[
    q \defeq 1_{ \Mult(\Contz(\base; \A|_{\base}))} \cdot \norm{h
    }^2_\infty - (h^** h) \quad \text {and } \quad q_{\lambda} =
    u_{\lambda} q u_{\lambda}
  \]
  where \(q\in \Mult(\Contz(\base; \A|_{\base}))\) and
  \(q_{\lambda}\in \Contz(\base; \A|_{\base})\). The following
  statements hold:
  \begin{enumerate}
  \item \(q\) is a positive element
    in~\(\Mult(\Contz(\base; \A|_{\base}))\).
  \item \(q_{\lambda}\) is a positive element
    in~\(\Contz(\base; \A|_{\base})\) for \(\lambda \in \Index\).
  \item For a fixed \(v\in \Hils_{00}\),
    \[
      \lim_{\lambda} (\norm{L(\norm{h}_\infty u_{\lambda})v}^2 -
      \norm{ L(h*u_{\lambda})v}^2) = (\norm{h}_\infty \norm{v})^2 -
      \norm{ L(h)v}^2.
    \]
  \end{enumerate}
\end{lemma}

\begin{proof}
  (1):
  Lemma~\ref{lem:convo-func:supp:bisec}(\ref{lem:inv-with-adj})--(\ref{lem:uni-norm-bisec})
  implies that \(h^**h\in \Contc(\base; \A|_{\base})\) with
  \( \norm{h^* * h}_\infty = \norm{h}_\infty^2 \).  Now
  Lemma~\ref{lem:positive-element} implies
  that~\(q = 1_{ \Mult(\Contz(\base; \A|_{\base}))} \cdot
  \norm{h^**h}_\infty -(h^**h) \) is a positive element
  in~\(\Mult(\Contz(\base; \A|_{\base}))\).

  \noindent (2): Let \(k\) be the positive square root of~\(q\) in the
  multiplier algebra. Then
  \(u_{\lambda}qu_{\lambda} = (k u_{\lambda})^*(ku_{\lambda})\).

  \noindent (3): For~\(v\in \Hils_{00}\), compute the difference
  \begin{multline*}
    \norm{L(\norm{h}_\infty u_{\lambda})v}^2 - \norm{
      L(h*u_{\lambda})v}^2 \\= \big\langle L(\norm{h}_\infty
    u_{\lambda})v \mathbin{,} L(\norm{h}_\infty u_{\lambda})v
    \big\rangle - \binpro{ L(h*u_{\lambda})v}{ L(h*u_{\lambda})v}.
  \end{multline*}

  \noindent As~\(v\in \Hils_{00}\), we can write
  \(v=\sum_{i=1}^k L(f_i)\xi_i\) for some \(f_i\in \Contc(G;\A)\),
  \(\xi_i\in \Hils_0\) and \(k\in \N\). Substitute this value of~\(v\)
  in last subtraction, and compute it further:
  \begin{multline*}
    \binpro{L(\norm{h}_\infty u_{\lambda}) \sum_{i=1}^k L(f_i)\xi_i}{
      L(\norm{h}_\infty u_{\lambda}) \sum_{i=1}^k L(f_i)\xi_i} \\ -
    \binpro{L(h*u_{\lambda}) \sum_{i=1}^k L(f_i)\xi_i
    }{L(h*u_{\lambda})\sum_{i=1}^k
      L(f_i)\xi_i}\\
    = \sum_{i,j} \Bigl(\norm{h}_\infty^2\binpro{L(u_{\lambda} *
      f_i)\xi_i}{L(u_{\lambda}*f_j)\xi_j} - \binpro{ L(h*u_{\lambda}*
      f_i)\xi_i}{ L(h*u_{\lambda} * f_j)\xi_j}\Bigr).
  \end{multline*}
  Using Equations~\eqref{obs:pre-rep-computation-obs-1}
  and~\eqref{obs:pre-rep-computation-obs-3}, the limit of preceding
  subtraction over~\(\Index\) can be computed to be

  \begin{multline*}
    \sum_{i,j} \Big(\norm{h}_\infty^2 \binpro{L(f_i)\xi_i}
    {L(f_j)\xi_j} -
    \binpro{ L(h*f_i)\xi_i}{ L(h*f_j)\xi_j}\Bigr)\\
    = \sum_{i,j} \Big(\norm{h}_\infty^2 \binpro{L(f_i)\xi_i}
    {L(f_j)\xi_j} -
    \binpro{ L(h) L(f_i)\xi_i}{ L(h)L(f_j)\xi_j}\Bigr)\\
    = \norm{h}_\infty^2 \Binpro{\sum_i^kL(f_i)\xi_i}
    {\sum_i^kL(f_i)\xi_i} - \Binpro{ L(h) \sum_i^kL(f_i)\xi_i}{ L(h)
      \sum_i^kL(f_i)\xi_i}
  \end{multline*}
  which proves the claim.
\end{proof}

\begin{proposition}
  \label{prop:pre-rep-to-rep-of-C0-X}
  Suppose that a Fell bundle \(p\colon \A \to G\) over an \etale\
  groupoid~\(G\) is given.  Suppose that a pre-representation
  \(L\colon \Contc(G;\A) \to \mathrm{Lin}(\Hils_0)\) is given where
  \(\Hils_0 \subseteq \Hils\) is a dense subspace. Then~\(L\) can be
  extended to a
  representation~\(M\colon \Contc(G, \A) \to \Bound(\Hils)\) that is
  continuous in the inductive limit topology. To be precise, the
  following holds: let \(h\in \Contc(G; \A)\) and~\(\Hils_{00}\) have
  the same meaning as in Definition~\ref{def:pre-rep}. For this~\(h\),
  define the linear operator~\(M'(h)\) on the inner product
  space~\(\Hils_{00}\) which is given on generating
  vectors~\(L(f)\xi\in \Hils_{00}\) by
  \begin{equation*}
    M'(h)(L(f)\xi) = L( h * f)\xi\label{eq:rep-from-pre-rep};
  \end{equation*}
  here \(f\in \Contc(G; \A)\) and \(\xi\in \Hils_0\). Then:
  \begin{enumerate}
  \item the linear operator \(M'(h)\) extends to a bounded linear
    operator \(M(h)\colon \Hils\to \Hils\).
  \item \(M\colon \Contc(G;\A) \to \Bound(\Hils)\) given by
    \(M\colon h \mapsto M(h)\) is a nondegenerate representation that
    is continuous in the inductive limit topology on~\(\Contc(G;\A)\).
  \item \(M\) is an extension of~\(L\), that is,
    \(M(h)|_{\Hils_{0}} = L(h)\).
  \end{enumerate}
\end{proposition}
\begin{proof}
  \noindent (1): We prove it in two steps.  \medskip
  
  \noindent Step~I: Assume that \(\supp(h)\) lies
  inside an open bisection. We intend to prove that
  \begin{equation}
    \norm{M'(h)v}\leq \norm{h}_\infty \norm{v} \quad \text{ for }
    v \in \Hils_{00}.\label{eq:bdd-op-on-M00}
  \end{equation}
  If last inequality is proved, then due to the density of
  \(\Hils_{00}\subseteq \Hils\), \(M'(h)\) can be extended to a
  bounded operator \(M(h)\colon \Hils \to \Hils\).
  
  For proving Equation~\eqref{eq:bdd-op-on-M00}, consider the given
  approximate unit~\((u_{\lambda})_{\lambda \in \Index}\)
  of~\(\Contz(\base;\A|_{\base})\) where
  \(u_{\lambda}\in \Contc(\base;\A|_{\base})\)
  (cf.~Corollary~\ref{cor:cc-apprx-unit-general-case}).
  Let~\(v\in \Hils_{00}\), so that~\(v = \sum_{i=1}^k L(f_i)\xi_i\)
  for some~\(k\in \N\), \(f_i\in \Contc(G;\A)\)
  and~\(\xi_i\in \Hils_0\).  Note that Lemma~\ref{lem:sqrt-of-q}(3)
  implies that
  \[
    (\norm{h}_\infty \norm{v})^2 - \norm{M'(h)v}^2 = \lim_{\lambda}
    (\norm{h}_\infty \norm{L(u_{\lambda})v})^2 -
    \norm{L(h*u_{\lambda})v}^2.
  \]
  Therefore, for proving Equation~\eqref{eq:bdd-op-on-M00}, it
  suffices to show that
  \((\norm{h}_\infty \norm{L(u_{\lambda})v})^2 -
  \norm{L(h*u_{\lambda})v}^2 \geq 0\) for all \(\lambda \in
  \Index\). For \(\lambda \in \Index\), we compute this differences as
  follows:
  \begin{multline*}
    (\norm{h}_\infty \norm{L(u_{\lambda})v})^2 -
    \norm{L(h*u_{\lambda})v}^2 = \norm{h}^2_\infty
    \norm{\sum_{i=1}^k L(u_{\lambda}*f_i)\xi_i}^2 - \norm{ \sum_{i=1}^k L(h*u_{\lambda} * f_i)\xi_i}^2\\
    = \Binpro{\norm{h}^2_\infty \left(\sum_{i=1}^k
        L(u_{\lambda}*f_i)\xi_i\right)}{\sum_{i=1}^k L(u_{\lambda}
      *f_i)\xi_i} - \Binpro{ \sum_{i=1}^k L(h * u_{\lambda}*
      f_i)\xi_i}{ \sum_{i=1}^k L(h * u_{\lambda}* f_i)\xi_i}.
  \end{multline*}

  \noindent Now use the adjointability of the pre-representation, so
  that the last term becomes ~
  \begin{multline}\label{eq:pre-rep-to-rep-of-C0-X-1}
    \sum_{i, j=1}^k \Binpro{ ( L(\norm{h}^2_\infty
      u_{\lambda}*u_{\lambda}* f_i))\xi_i} {L(f_j)\xi_j}\\ - \sum_{i,
      j=1}^k \Binpro{ L(u_{\lambda}*h^* * h * u_{\lambda} *
      f_i)\xi_i}{L(f_j)\xi_j }.
  \end{multline}

  \noindent Notice that \(h^**h \in \Contc(\base;\A|_{\base})\). As a
  consequence
  \(u_{\lambda}*(h^**h) *u_{\lambda} = u_{\lambda}\cdot (h^**h) \cdot
  u_{\lambda}\in \Contc(\base;\A|_{\base})\).  Moreover,
  Equation~\eqref{eq:conv} implies that the convolutions
  \begin{align*}
    ( u_{\lambda}*(h^**h)*u_{\lambda}) * f_i &= (u_{\lambda}\cdot (h^**h) \cdot u_{\lambda})\circ r \cdot
                                               f_i = (u_{\lambda}\circ r \cdot
                                               (h^**h)\circ r \cdot
                                               u_{\lambda}\circ r)
                                               \cdot f_i\; ;\\
    \text{and }\quad    u_{\lambda}^2 * f_i &= (u_{\lambda}^2 \circ r) \cdot f_i = (u_{\lambda}
                                              \circ r)^2 \cdot f_i.
  \end{align*}
  \noindent
  Therefore, the difference~\eqref{eq:pre-rep-to-rep-of-C0-X-1} can be
  written as
  \begin{multline*}
    \sum_{i, j=1}^k \Binpro{L( \norm{h}^2_\infty \cdot
      (u_{\lambda}\circ r)^2\cdot f_i)\xi_i}{L(f_j)\xi_j} \\- \sum_{i,
      j=1}^k \Binpro{L( (u_{\lambda}\circ r)\cdot (h^**h)\circ r
      \cdot(u_{\lambda}\circ r) \cdot f_i)\xi_i}{L(f_j)\xi_j}
  \end{multline*}
  Using the bilinearity of the inner product, the last term can be
  written as
  \begin{multline*}
    \sum_{i, j=1}^k \Binpro{L\left(\left(\norm{h}^2_\infty \cdot
          (u_{\lambda}\circ r)^2 - (u_{\lambda}\circ r)\cdot
          (h^**h)\circ r \cdot(u_{\lambda}\circ r)
        \right)\cdot f_i\right)\xi_i}{L(f_j)\xi_j} \\
    = \sum_{i, j=1}^k \Binpro{L\left(\left(u_{\lambda}\circ r
          \left(\norm{h}^2_\infty\cdot 1_{
              \Mult(\Contz(\base;\A|_{\base}))} - (h^**h) \circ r
          \right) u_{\lambda}\circ r\right) \cdot
        f_i\right)\xi_i}{L(f_j)\xi_j}
  \end{multline*}
  where
  \(1_{\Mult(\Contz(\base); \A|_{\base})}\in \Mult(\Contz(\base);
  \A|_{\base})\) is the unit in the multiplier algebra. Following
  Lemma~\ref{lem:sqrt-of-q}, put
  \(q =\norm{h}^2_\infty\cdot 1_{ \Mult(\Contz(\base;\A|_{\base}))} -
  (h^**h) \in \Mult(\Contz(\base); \A|_{\base})\) and
  \(q_{\lambda} = u_{\lambda} q u_{\lambda}\in
  \Contc(\base;\A|_{\base})\). Then the last term in above computation
  becomes
  \[
    \sum_{i, j=1}^k \binpro{L\left((q_{\lambda} \circ r) \cdot
        f_i\right)\xi_i}{L(f_j)\xi_j}.
  \]
  As Lemma~\ref{lem:sqrt-of-q} says, \(q_{\lambda}\) is a positive
  element
  in~\(\Contc(\base;\A|_{\base}) \subseteq
  \Contz(\base;\A|_{\base})\).
  Let~\(s_{\lambda}\in \Contz(\base;\A|_{\base})\) be its unique
  positive square root;
  then~\(s_{\lambda}\in \Contc(\base;\A|_{\base})\). By substituting
  \(q_{\lambda}=s_{\lambda}^2\) last term of main computation becomes
  \begin{multline*}
    \sum_{i, j=1}^k \binpro{L((s_{\lambda}\circ r) \cdot
      f_i)\xi_i}{L((s_{\lambda}\circ r) \cdot f_j)\xi_j} = \sum_{i,
      j=1}^k
    \binpro{L(s_{\lambda}* f_i)\xi_i}{L(s_{\lambda} * f_j)\xi_j}\\
    = \binpro{\sum_{i=1}^nL(s_{\lambda}*
      f_i)\xi_i}{\sum_{i=1}^nL(s_{\lambda} * f_i)\xi_i} \geq 0.
  \end{multline*}
  
  \noindent Thus Equation~\eqref{eq:bdd-op-on-M00} is proved.
  \smallskip

  \noindent Step~II: Now, suppose \(h\in \Contc(G, \A)\) is given. Write \(h= \sum_{i=1}^p h_i\) where each \(h_i\)
  is supported in a bisection. Now, Equation~\eqref{eq:bdd-op-on-M00}
  implies that
  \(\norm{M'(h)v} \leq (\sum_{i=1}^p \norm{h_i}_\infty ) \norm{v} \leq
  p \norm{h}_\infty \norm{v}\) for \(v\in \Hils_{00}\); here the
  number~\(p\) depends on a representation of~\(h\) as a sum
  of~\(h_i\)'s. But for the fixed representation,~\(p\) is fixed. Now
  we can extend~\(M'(h)\) to a bounded
  operator~\(M(h)\colon \Hils \to \Hils\) using the density of
  \(\Hils_0\subseteq \Hils\).

  Finally, note that \(M(h)\leq p\norm{h}_\infty\). Therefore,
  \(M\colon \Contc(G;\A) \to \Bound(\Hils)\) is continuous in the
  inductive limit topology.  \medskip

  \noindent (3): By construction of~\(M(h)\) in~(1) of the current
  lemma, it is clear that \(M(h)|_{\Hils_{00}} =
  L(h)|_{\Hils_{00}}\). Since \(\Hils_{00}\subseteq \Hils_0\) is
  dense, \(M(h)|_{\Hils_0} = L(h)\).  \medskip
  
  \noindent (2): Let \(h\in \Contc(G;\A)\). Then due to~(3) and the
  adjointability of a pre-representation, \(M(h^*) = M(h)^*\).
  Finally, the nondegeneracy of~\(M\) follows from the nondegeneracy
  of a pre-representation and the fact that the linear subspace
  of~\(\Contc(G;\A)\) spanned by \(\Contc(G;\A)*\Contc(G;\A)\) is
  uniformly dense, cf.~Lemma~\ref{lem:mult-by-bdd-sections}(2).
\end{proof}

For the Fell bundle \(p\colon \A \to G\), define the following
collections:
\begin{enumerate}[leftmargin=*]
\item \(\mathcal{R}_R\), the collection of representations;
\item \(\mathcal{R}_{H}\), the collection of \(I\)\nb-norm bounded
  representations;
\item \(\mathcal{R}_{\mathrm{In}}\) the collection of representations
  continuous in the inductive limit topology;
\item \(\mathcal{R}_P\), the collection of all representations induced
  from pre-representations as in
  Proposition~\ref{prop:pre-rep-to-rep-of-C0-X}.
\end{enumerate}
With this notation, we can state the following theorem:

\begin{theorem}\label{thm:norm-equi}
  Let \(p\colon \A \to G\) be a (not necessarily saturated) Fell
  bundle over an \etale~groupoid \(G\) (not necessarily
  Hausdorff). For \(f\in \Cc(G;\A)\), define\footnote{For ``RRR''
    fans, \(\rho\in \mathcal{R}_R\) is not a coincidence!}
  \begin{align*}
    \norm{f}_R &= \sup_{\rho\in \mathcal{R}_R}\{\norm{\rho(f)} \}, & \norm{f}_{H}
    &= \sup_{\rho\in \mathcal{R}_{H}}\{\norm{\rho(f)} \},\\
    \norm{f}_{\mathrm{In}} &= \sup_{\rho\in \mathcal{R}_{\mathrm{In}}}\{\norm{\rho(f)} \}, & \norm{f}_P &= \sup_{\rho\in \mathcal{R}_P}\{\norm{\rho(f)} \}.
  \end{align*}
  
  \noindent Then
  \(\norm{f}_R = \norm{f}_{H} = \norm{f}_{\mathrm{In}}=
  \norm{f}_{P}\).  And \(\norm{f} \defeq \norm{f}_R\) is a
  \(\Cst\)\nb-norm on \(\Contc(G;\A)\). Moreover, the
  \(\Cst\)\nb-algebra \(\overline{\Contc(G;\A)}^{\norm{\cdot}_R}\) is
  same as the \(\Cst\)\nb-algebra that Muhly--Williams define
  in~\cite{Muhly-Williams2008FellBundle-ME} for saturated Fell bundle
  over locally compact Hausdorff second countable groupoid.
\end{theorem}

\begin{proof}
  The equality
  \(\norm{f}_R =\norm{f}_{H} = \norm{f}_{\mathrm{In}} = \norm{f}_P\)
  follows from equivalence of the four notions of representations,
  Propositions~\ref{prop:three-reps-are-equi}
  and~\ref{prop:pre-rep-to-rep-of-C0-X}. The claim that
  \(\norm{\cdot}_R\) is a \(\Cst\)\nb-norm, follows from a standard
  argument, for example, see~\cite[Theorem
  9.2.3]{Sims-Szabo-Williams2020Operator-alg-book}. Muhly and Williams
  use~\(\norm{\cdot}_{\mathrm{In}}\) and~\(\norm{\cdot}_P\) to define
  Fell bundle \(\Cst\)\nb-algebra; we show that these norms agree
  with~\(\norm{\cdot}_R\); this justifies the last claim.
\end{proof}

The \(\Cst\)\nb-algebra \(\overline{\Contc(G;\A)}^{\norm{\cdot}}\) is
the \emph{full} \(\Cst\)\nb-algebra of \(p\colon \A\to G\). This
\(\Cst\)\nb-algebra is denoted by~\(\Cst(G;\A)\).

Assume that in Theorem~\ref{thm:norm-equi}, \(G\) is locally compact
Hausdorff and second countable \etale; and \(\A\) saturated and
separable. Then Muhly and Williams define the full \(\Cst\)\nb-algebra
of~\(\A\) using~\(\mathcal{R}_H, \mathcal{R}_{\mathrm{In}}\) and
\(\mathcal{R}_P\) (and unitary representations of~\(p\)). They show
that these three norms are same. We have added one more
norm~\(\mathrm{R}_R\). The \(\Cst\)\nb-algebra given by last theorem
is same as that of Muhly and Williams.

For all following remarks a Fell bundle \(p\colon \A \to G\) over an
\etale\ groupoid is fixed.

\begin{remark}
  Consider Theorem~\ref{thm:norm-equi} for the trivial Fell bundle
  over an \etale\ groupoid. Then \(\mathcal{R}_R\) is the norm on
  \(\Contc(G)\) that Exel defines
  in~\cite{Exel2008Inverse-semigroup-combinotorial-Cst-alg};
  \(\mathcal{R}_I\) is discussed in~\cite{Renault1980Gpd-Cst-Alg} by
  Renault or~\cite{Sims-Szabo-Williams2020Operator-alg-book} by Sims;
  and \(\mathcal{R}_{\mathrm{In}}\) and \(\mathcal{R}_{P}\) are the
  norms Renault defines
  in~~\cite{Renault1985Representations-of-crossed-product-of-gpd-Cst-Alg}
  (and proves are equal). Theorem~\ref{thm:norm-equi} shows that all
  these norms, and hence corresponding \(\Cst\)\nb-algebras are same.
\end{remark}

\begin{remark}\label{rem:section-alg-of-base}
  Now \(\Contc(\base;\A|_{\base})\) is a \Star-algebra of
  \(\Contc(G;\A)\). Due to the uniqueness of \(\Cst\)\nb-norm,
  \(\overline{\Contc(\base;\A|_{\base})}^{\norm{\cdot}} \subseteq
  \Cst(G,\A)\) must be~\(\Contz(\base;\A|_{\base})\). Thus,
  the~\(\Contz(\base)\)-algebra \(\Contz(\base;\A|_{\base})\) is a
  \(\Cst\)\nb-subalgebra of \(\Cst(G;\A)\). In particular, for
  \(f\in \Contz(\base;\A|_{\base})\), the \(\Cst\)\nb-norm
  \(\norm{f} = \norm{f}_\infty\).
\end{remark}

\begin{remark}
  For \(f\in \Contc(G;\A)\) be supported in a bisection, the
  \(\Cst\)\nb-norm
  \[
    \norm{f}^2 = \norm{f*f^*} = \norm{f*f^*}_\infty =
    \norm{f}_\infty^2
  \]
  where the first equality is the \(\Cst\)\nb-identity, the second
  equality is because \(f*f^* \in \Contz(\base; \A|_{\base})\) and the
  last one is Lemma~\ref{lem:convo-func:supp:bisec}(4). Moreover,
  Example~\ref{exa:I-norm-on-bisection} says that
  \(\norm{f}_I = \norm{f}_\infty\). Therefore, for a section supported
  on a bisection, the full \(\Cst\)\nb-norm equals the \(I\)\nb-norm
  or uniform norm.
\end{remark}

\begin{remark}
  Let \((u_\lambda)\) be an approximate unit of
  \(\Contz(\base; \A|_{\base})\) consisting of compactly supported
  sections. Remark~\ref{rem:section-alg-of-base} and
  Lemma~\ref{lem:mult-by-bdd-sections}(1) imply that~\((u_\lambda)\)
  is also an approximate unit of~\(\Cst(G;\A)\).
\end{remark}

\begin{example}
  \label{exa:Cst-bundle-over-unit}
  Let \(G\) be locally compact Hausdorff \etale\ groupoid and
  \(\A\to \base\) an upper semicontinuous bundle of
  \(\Cst\)\nb-algebras. Consider the bundle \(\B\to G\) which is
  \(\A\) over~\(\base\) and has the zero fibres otherwise. Then~\(\B\)
  is upper semicontinuous, and, in fact, a Fell bundle over~\(G\). The
  Fell bundle~\(\B\) is not a saturated bundle unless~\(\A\) consists
  of the zero \(\Cst\)\nb-algebras. In this case,
  \(\Contc(G; \B) = \Contc(\base; \A)\) and
  \(\Cst(G; \B) \iso \Contz(\base; \A)\).
\end{example}

\begin{example}
  Let \(G\) be the groupoid of trivial equivalence relation on the set
  \(\{0, 1\}\), that is, \(G = \{(0, 0), (0,1), (1,0),(1,1)\}\). Then
  \(G\) is an \etale\ groupoid.  Let \(A\) and \(B\) be two
  \(\Cst\)\nb-algebras and \(X\) a (not necessarily full) Hilbert
  \(A\)-\(B\)\nb-bimodule. Let \(X^*\) be the dual module of \(X\).

  Consider the Fell bundle \(p\colon \A \to G\) as follows: the fibres
  are \(\A_{(0,0)} = A, \A_{(0,1)} = X, \A_{(1,0)} = X^*\) and
  \(\A_{(1,1)} = B\). The multiplications of Fell bundle
  \(\A_{(0,1)} \times \A_{(1,0)} \to \A_{(0,0)}\) is defined by
  \(\xi \cdot \eta \defeq {}_A \inpro{\xi}{\eta^*}\) and
  \(\A_{(1,0)} \times \A_{(0,1)} \to \A_{(1,1)}\) is defined by
  \(\eta \cdot \xi \defeq \inpro{\eta^*}{\xi}_B\), where \(\xi \in X\)
  and \(\eta \in X^* \). This Fell bundle is unsaturated unless \(X\)
  is an \emph{imprimitivity bimodule}.  Let \(L\) be the linking
  algebra associated with the Hilbert bimodule~\(X\) defined by Brown,
  Mingo and
  Shen~\cite{Brown-Mingo-Shen1994Quasi-mult-enbed-Hilbert-mod}.
  Define \(\phi \colon \mathrm{C}(G;\A) \to L\) by
  \[
    \phi(f) = \begin{pmatrix}
      f(0,0) & f(0,1) \\
      f(1,0) & f(1,1)
    \end{pmatrix}
  \]
  for \(f\in \mathrm{C}(G;\A)\). Then \(\phi\) is a bijective
  \(^*\)\nb-homomorphism. Since \(L\) is a \(\Cst\)\nb-algebra, by the
  uniqueness of the \(\Cst\)\nb-norm \(\Cst(G;\A) \iso L\). This
  example is an \emph{unsaturated} version of~Kumjian's
  example~\cite[Example 3.5(ii)]{Kumjian1998Fell-bundles-over-gpd}.
\end{example}

\begin{example}[Fell bundle associated with partial
  actions]\label{exa:part-act}
  One may associate Fell bundles to partial actions of groups or
  groupoids; and these Fell bundles are the prototypical examples of
  unsaturated Fell bundles. Assume that a discrete group~ \(\Gamma\)
  acts partially on a unital \(\Cst\)\nb-algebra \(A\)
  (see~\cite[Chapter~11]{Exel2017PDS-Fell-bundles}). Then Exel
  proved~\cite[Chapter 16]{Exel2017PDS-Fell-bundles} that the Fell
  bundle \(\B\) associated with the partial action is an
  \emph{unsaturated} one. Moreover, the \(\Cst\)\nb-algebra of the
  Fell bundle \(\B\) is naturally isomorphic to the partial crossed
  product \(A\rtimes \Gamma\) (Proposition~16.28
  of~\cite{Exel2017PDS-Fell-bundles}).

  Anantharaman-Delaroche generalises Exel's idea of constructing a
  Fell bundle out of group partial action~\cite[Chapter
  16]{Exel2017PDS-Fell-bundles} to groupoids
  in~\cite[Section~4]{Anantharaman-D:2020Partial-Action-of-Gpd}. In
  particular, let \(G\) be a locally compact Hausdorff \etale\
  groupoid with unit space \(\base\). Let~\(A\) be a
  \(\Contz(\base)\)\nb-algebra on which~\(G\) is acting
  partially. In~\cite[Section~4]{Anantharaman-D:2020Partial-Action-of-Gpd},
  Anantharaman-Delaroche
  defines the Fell
  bundle \(p\colon \B\to G\) associated with this action. She
  constructs~\(\Cst(G;\B)\) using the \(I\)\nb-norm bounded
  representation. Since we have assumed~\(G\) to be {\etale},
  her~\(\Cst(G;\B)\) is same as the Fell bundle \(\Cst\)\nb-algebra
  Theorem~\ref{thm:norm-equi} leads to. As Anantharaman-Delaroche
  notices~\cite[Remark~4.5]{Anantharaman-D:2020Partial-Action-of-Gpd},
  the Fell bundle~\(\Cst\)\nb-algebra is isomorphic to the
  \(\Cst\)\nb-algebra of the partial action~\(G\curvearrowright
  A\). One can note that her construction and last remark hold even if
  the \etale{} groupoid~\(G\) is not Hausdorff as in our sense.
\end{example}

\end{document}